\documentclass[reqno,12pt]{amsart}
\usepackage{mathrsfs}
\usepackage{amscd}
\usepackage{amsmath}
\usepackage{latexsym}
\usepackage{amsfonts}
\usepackage{amssymb}
\usepackage{amsthm}
\usepackage{graphicx}

\numberwithin{equation}{section}

\def\H{{\cal H}}

\def\R{\mathbb{R}}

\def\T{\mathbb{T}}

\def\H1{H^1(\R)}

\newtheorem{thm}{Theorem}[section]
\newtheorem{lem}{Lemma}[section]

\makeatletter
\newcommand{\Extend}[5]{\ext@arrow0099{\arrowfill@#1#2#3}{#4}{#5}}
\makeatother

\begin{document}

\setcounter{page}{1}

\title[Global Well-posedness for DNLS]{Global well-posedness for the nonlinear\\
Schr\"{o}dinger equation with derivative \\
in energy space}

\author{Yifei Wu}
\address{School of  Mathematical Sciences, Beijing Normal University, Laboratory of Mathematics and Complex Systems,
Ministry of Education, Beijing 100875, P.R.China}
\email{yifei@bnu.edu.cn}
\thanks{The author was partially supported
by the NSF of China (No. 11101042), and the Fundamental Research Funds for the Central Universities of China.}




\maketitle

\begin{abstract}\noindent
In this paper, we prove that there exists some small
$\varepsilon_*>0$, such that the derivative nonlinear
Schr\"{o}dinger equation (DNLS) is global well-posedness in the
energy space, provided that the initial data $u_0\in \H1$ satisfies
$\|u_0\|_{L^2}<\sqrt{2\pi}+\varepsilon_*$. This result shows us that
there are no blow up solutions whose masses slightly exceed $2\pi$,
even if their energies are negative.  This phenomenon is much different from the behavior
of nonlinear Schr\"odinger equation with critical nonlinearity.  The technique used is a variational argument together with the
momentum conservation law.
Further, for the DNLS on half-line $\R^+$, we show the blow-up for the solution with negative energy.
\end{abstract}

\section{Introduction}
We study the following Cauchy problem of the nonlinear
Schr\"{o}dinger equation with derivative (DNLS):
 \begin{equation}\label{eqs:DNLS}
   \left\{ \aligned
    &i\partial_{t}u+\partial_{x}^{2}u=i\lambda\partial_x(|u|^2u),\qquad t\in \R, x\in \R,
    \\
    &u(0,x)  =u_0(x)\in   H^1(\mathbb{R}),
   \endaligned
  \right.
 \end{equation}
where $\lambda\in \R$.  It arises from studying the propagation of
circularly polarized Alfv\'{e}n waves in magnetized plasma with a
constant magnetic field, see \cite{MOMT-PHY, M-PHY, SuSu-book} and the references therein.

This equation is $L^2$-critical in the sense that both the equation
and the $L^2$-norm are invariant under the scaling transform
\begin{equation*}
u_\alpha(t,x) = \alpha^{\frac
12} u(\alpha^2 t, \alpha x),\quad \alpha>0.
\end{equation*}
It has the same scaling invariance as the quintic nonlinear
Schr\"{o}dinger equation,
\begin{equation*}
i\partial_{t}u+\partial_{x}^{2}u+\mu |u|^4u=0, \quad t\in \R,x\in\R,
\end{equation*}
and the quintic generalized Korteweg-de Vries equation,
\begin{equation*}
    \partial_t u +  \partial^3_{x} u +\mu \partial_x(u^5)=0,\quad t\in \R,x\in\R.
 \end{equation*}

One may always take $\lambda=1$ in (\ref{eqs:DNLS}), since the general case can be reduced to this case by the following two transforms. First, we
apply the transform
$$
u(t,x)\mapsto\bar{u}(-t,x),
$$
then reduce the equation to the case of $\lambda>0$. Then we take the rescaling transform
$$
u(t,x)\mapsto\frac{1}{\sqrt{\lambda}}u(t,x)
$$
and reduce it to the case of $\lambda=1$.
So in this sense, the equation (\ref{eqs:DNLS}) can always be regarded as the focusing
equation. From now on, we always assume
that $\lambda=1$ in (\ref{eqs:DNLS}).

The $H^1$-solution of (\ref{eqs:DNLS}) obeys the following three conservation laws. The first one is the conservation of the mass
\begin{equation}\label{mass-law}
M(u(t)):=\int_{\R}|u(t)|^2\,dx=M(u_0),
\end{equation}
the second one is the conservation of energy
\begin{equation}\label{energy-law}
E_D(u(t)):= \int_{\R}\Big(| u_x(t)|^2+\frac{3}{2}\text{Im}
|u(t)|^2u(t)\overline{u_x(t)}+\frac{1}{2}|u(t)|^6\Big)\,dx=E_D(u_0),
\end{equation}
the third one is the conservation of momentum (see \eqref{mom-law} below)
\begin{equation}
P_D(u(t)):=\mbox{Im} \int_{\R}\bar u(t) u_x(t)\,dx-\frac 12\int_{\R}
|u(t)|^4\,dx=P_D(u_0).
\end{equation}

Local well-posedness for the Cauchy problem (\ref{eqs:DNLS}) is
well-understood. It was proved for the energy space $H^1(\R)$ by Hayashi and Ozawa in \cite{Ha-93-DNLS, HaOz-92-DNLS,
HaOz-94-DNLS}, see also Guo and Tan \cite{GuTa91} for earlier result
in smooth spaces. For rough data below the energy space, Takaoka \cite{Ta-99-DNLS-LWP} proved
the local well-posedness in $H^s(\R)$ for $s\geq 1/2$.This result was shown to be sharp in the
sense that the flow map fails to be uniformly $C^0$ for $s<1/2$, see
Biagioni and Linares \cite{BiLi-01-Illposed-DNLS-BO} and Takaoka
\cite{Ta-01-DNLS-GWP}.

The global well-posedness for (\ref{eqs:DNLS}) has been also widely
studied.
By using mass and energy conservation laws, and by developing the gauge
transformations, Hayashi and Ozawa \cite{HaOz-94-DNLS, Oz-96-DNLS}
proved that the problem (\ref{eqs:DNLS}) is globally well-posed in energy space
$H^1(\R)$ under the condition
\begin{equation}\label{small condition}
\|u_0\|_{L^2}<\sqrt{2\pi }.
\end{equation}
Further, for initial data of regularity below the energy space, Colliander et al \cite{CKSTT-01-DNLS, CKSTT-02-DNLS} proved the
global well-posedness for (\ref{eqs:DNLS}) in $H^s(\R)$ for $s>\frac
12$, under the condition (\ref{small condition}). Recently, Miao, Wu and Xu \cite{Miao-Wu-Xu:2011:DNLS} proved that (\ref{eqs:DNLS}) is globally
well-posed in the critical space $H^{\frac 12}(\R)$, also under the
condition (\ref{small condition}). For other works on the DNLS in the periodic case, see a few of examples \cite{Grhe-95,Herr,NaOhSt-12,Win}.

As is mentioned above, all the results on global existence for initial
data were obtained under the assumption of \eqref{small condition}.
Since $\sqrt{2\pi }$ is just the mass of the ground state of the corresponding elliptic problem,
the condition \eqref{small condition} was naturally used to keep the energy positive; see \cite{CKSTT-01-DNLS, Miao-Wu-Xu:2011:DNLS} for examples.
Now one may wonder what happens to the well-posedness for the solution when  \eqref{small condition} is not fulfilled.
Our first main result in this paper is to improve the assumption \eqref{small condition} and obtain the global well-posedness as follow.
\begin{thm}\label{thm:main1}  There exists a small $\varepsilon_*>0$ such that for any $u_0\in \H1$
with
\begin{equation}\label{small condition-further}
\int_{\R}|u_0(x)|^2\,dx <{2\pi }+\varepsilon_*,
\end{equation}
the Cauchy problem (\ref{eqs:DNLS}) ($\lambda=1$) is globally
well-posed in $H^1(\R)$ and the solution $u$ satisfies
$$
\|u\|_{L^\infty_t H^1_x}\le C(\varepsilon_*, \|u_0\|_{H^1}).
$$
\end{thm}

The technique used to prove Theorem \ref{thm:main1} is a
variational argument together with the momentum and energy
conservation laws. The key ingredient is the momentum conservation law, rather than
the energy
conservation law, upon which many (subcritical) problems rely when studying the global existence.
We argue for contradiction. Suppose that the solution of \eqref{eqs:DNLS} blows up at finite/infinite time $T$ and $t_n$ is a time sequence tending to $T$ such that $u(t_n)$ tends to infinity in $\H1$ norm. Then, thanks to the energy conservation law and a variational lemma from Merle \cite{Me-2001-gKdV}, $u(t_n)$ is close to the ground state $Q$ (see below for its definition) up to a spatial transformation, a  phase rotation and a scaling transformation. On the one hand, since $u(t_n)$ blows up at $T$, the scaling parameter $\lambda_n$ decays to zero; on the other hand, the conservation of momentum prevents $\lambda_n$ from tending to zero. This leads to a contradiction.

As mentioned above, Theorem \ref{thm:main1} improves the smallness of $L^2$-norm of the initial data of the
previous works on global existence (\cite{HaOz-94-DNLS, Oz-96-DNLS}). More importantly, it
reveals some special feature of the derivative nonlinear
Schr\"odinger equation. As discussed before, the smallness condition
(\ref{small condition}) in the previous works is imposed to
guarantee the positivity of the energy $E_D(u(t))$. Indeed, by using a
variant gauge transformation
\begin{equation}\label{gauge1}
v(t,x):=e^{-\frac{3}{4}i\int_{-\infty}^x |u(t,y)|^2\,dy}u(t,x),
\end{equation}
the energy is deduced to
\begin{equation}\label{energy-43}
E_D(u(t))=\|v_x(t)\|_{L^2_x}^2-\frac{1}{16}\|v(t)\|^6_{L^6_x}:=
E(v(t)),
\end{equation}
and then the positivity of $E(v)$ is followed by the sharp Gagliardo-Nirenberg inequality (see
\cite{W})
\begin{equation}\label{sharp Gagliardo-Nirenberg}
\|f\|_{L^6}^6\leq \frac{4}{\pi^2}\|f\|_{L^2}^4\|f_x\|_{L^2}^2.
\end{equation}
Once the mass is greater than $2\pi$, the positive energy can not be maintained.
To see this, we first
make use of the gauge transformation \eqref{gauge1}, and rewrite
 \eqref{eqs:DNLS} as
\begin{equation}\label{eqs:DNLS-under-=gauge1}
     i\partial_{t}v+\partial_{x}^{2}v=\frac i2|v|^2v_x-\frac i2v^2\bar{v}_x-\frac
     3{16}|v|^4v.
\end{equation}
There exists an obviously standing wave $e^{it}Q$ of
\eqref{eqs:DNLS-under-=gauge1}, where $Q$ is the unique (up to some
symmetries) positive solution of the following elliptic equation
$$
-Q_{xx}+Q-\frac 3{16}Q^5=0.
$$
This leads to the standing wave solution corresponding to the
equation \eqref{eqs:DNLS},
\begin{equation*}
   R(t,x):= e^{it+\frac{3}{4}i\int_{-\infty}^x
   Q^2\,dy}Q(x).
\end{equation*}
So on one hand, as a byproduct our result implies the stability of the standing wave solution, which has been proved by Colin and Ohta \cite{CoOh-06-DNLS}.
On the other hand,
$$
\|Q\|_{L^2}=\sqrt{2\pi},  \quad E(Q)=0,
$$
and the Fr\'echet derivation of the functional $ E(v)$ at $Q$ satisfies
$$
\delta E(Q)\cdot Q=-2\pi<0.
$$
These imply that there exists a $u_0$ such that $u_0$ obeys \eqref{small
condition-further} and $E_D(u_0)<0$. Therefore, there indeed exist global solutions
with negative energy, as stated in Theorem \ref{thm:main1}. Obviously this is much different from
the focusing, quintic nonlinear Schr\"{o}dinger equation \eqref{quintic-NLS} and focusing, quintic
generalized Korteweg-de Vries equation \eqref{gKdV-5}. For \eqref{quintic-NLS},
Ogawa and Tsutsumi \cite{OgTs92} proved that the solutions with the
initial data belonging to $\H1$ and negative energy must blow up in
finite time;  for \eqref{gKdV-5} Martel and Merle \cite{MaMe-2002-gKdV, Me-2001-gKdV} proved that the solutions with the
initial data belonging to $\H1$, negative energy and obeying some further decay
conditions blow up in finite time. In Section 3 below we will discuss some
differences among these three equations, in particular from the viewpoint of the
virial arguments.

Moreover, the situation of the Cauchy problem and the initial boundary value problem of the equation \eqref{eqs:DNLS} are much different. We consider the following Cauchy-Dirichlet problem of the nonlinear Schr\"odinger equation with derivative on half-line $\R^+$,
 \begin{equation}\label{DNLS-halfline}
   \left\{ \aligned
    &i\partial_{t}u+\partial_{x}^{2}u=i\partial_x(|u|^2u),\qquad t\in \R, x\in (0,+\infty),
    \\
    &u(0,x)  =u_0(x),
    \\
    &u(t,0)  =0.
   \endaligned
  \right.
 \end{equation}
We show that under some assumptions, the solution must blow up in finite time if its energy is negative.
\begin{thm}\label{thm:main2} Let  $u_0\in H^2(\R^+)$ and $xu_0\in L^2(\R^+)$, and let $u$ be the corresponding solution of \eqref{DNLS-halfline} which exists on the (right) maximal lifetime $[0,T_*)$. If $E(u_0)<0$,  then $T_*<\infty$. Moreover,
there exists a constant $C=C(u_0)>0$, such that
$$
\|u_x(t,x)\|_{L^2(\R^+)}
\ge \frac{C}{\sqrt{T_*-t}} \to \infty, \quad \mbox{ as } t\to  T^*-.
$$
\end{thm}

Lastly, we remark that it remains open for DNLS equation \eqref{eqs:DNLS}
whether there exists an $\H1$ initial data of much larger $L^2$-norm such that the corresponding solution blows up in finite time.
Moreover, it may be interesting to study the existence of global rough solutions when the condition
\eqref{small condition} on initial data is
relaxed.

This paper is organized as follow. In Section 2, we present the gauge transformation and prove the virial identities of DNLS. In Section 3, we
discuss the difference among the DNLS, the quintic NLS and the quintic gKdV equations. In Section 4, we study the initial boundary value problem of the DNLS on the half line and give the proof of Theorem \ref{thm:main2}. In Section 5, we
prove Theorem \ref{thm:main1}.

\vspace{0.5cm}
\section{Gauge transformations, Virial identities}

\subsection{Gauge transformations}
The gauge transformation is an important and very nice tool to study the nonlinear Schr\"oldinger equation with derivative (see Hayashi and Ozawa \cite{Ha-93-DNLS, HaOz-92-DNLS,
HaOz-94-DNLS}). It gives some improvement of the nonlinearity. In this subsection, we present the various gauge transformations and their properties. See, for examples \cite{CKSTT-01-DNLS, Oz-96-DNLS} for more details. We define
$$
\mathcal{G}_a u(t,x)=e^{ia\int_{-\infty}^x |u(t,y)|^2\,dy}u(t,x).
$$
Then
$
\mathcal G_a\mathcal G_{-a}=Id,
$
the identity transform. For any function $f$,
\begin{equation}\label{eqs:Gax}
\partial_x\mathcal{G}_a f=e^{ia\int_{-\infty}^x |f(t,y)|^2\,dy}\big(ia|f|^2f+f_x\big).
\end{equation}
Further, we have
\begin{lem} \label{lem:Ga}
If $u$ is the solution of \eqref{eqs:DNLS} (where $\lambda=1$), then $v=\mathcal{G}_a u$ is the solution of the equation,
$$
i\partial_{t}v+\partial_{x}^{2}v-i2(a+1)|v|^2v_x-i(2a+1)v^2\bar v_x+\frac12a(2a+1)|v|^4v=0.
$$
Moreover,
$$
E_D(u)=\big\|\partial_x\mathcal G_a u\big\|_2^2+\big(2a+\frac32\big)Im\int_{\R}|\mathcal G_a u|^2\mathcal G_a u\cdot\partial_x\overline{\mathcal G_a u}\,dx+\big(a^2+\frac32a+\frac12\big)\int_{\R}|\mathcal G_a u|^6\,dx.
$$
\end{lem}

The proof of this lemma follows from a direct computation and is omitted.

To understand how the gauge transform improves the nonlinearity in the present form \eqref{eqs:DNLS}, we introduce the following two transforms used in \cite{HaOz-94-DNLS, Oz-96-DNLS}. Let
$$
\phi=\mathcal G_{-1}u;\quad \psi=\mathcal G_{\frac12}\partial_x \mathcal G_{-\frac12}u,
$$
then $(\phi,\psi)$ solves the following system of nonlinear Schr\"odinger equation,
 \begin{equation}\label{syt:DNLS}
   \left\{ \aligned
    &i\partial_{t}\phi+\partial_{x}^{2}\phi=-i\phi^2\bar\psi,
    \\
    &i\partial_{t}\psi+\partial_{x}^{2}\psi=\psi^2\bar\phi.
   \endaligned
  \right.
 \end{equation}
Compared with the original equation \eqref{eqs:DNLS}, the system above has no loss of derivatives. Thus it is much more convenient to get the local solvability of \eqref{eqs:DNLS} for suitable smooth data by considering the system \eqref{syt:DNLS} instead.

As mentioned above, it is convenient to consider $v=\mathcal G_{-\frac34} u$. Then by Lemma \ref{lem:Ga}, the equation \eqref{eqs:DNLS} of $u$ reduces to \eqref{eqs:DNLS-under-=gauge1},
that is,
\begin{equation*}
     i\partial_{t}v+\partial_{x}^{2}v=\frac i2|v|^2v_x-\frac i2v^2\bar{v}_x-\frac
     3{16}|v|^4v.
\end{equation*}
Moreover, the energy $E_D(u)$ in \eqref{energy-law} is changed into $E(v)$ in \eqref{energy-43}.
In the sequel we shall consider \eqref{eqs:DNLS-under-=gauge1} and the energy \eqref{energy-43} of $v$ instead.

\subsection{Virial identities}
In this subsection, we discuss some virial identities for the nonlinear Schr\"odinger equation with derivative. Formally one may find that the virial quantity of $v$ is similar to that of mass-critical nonlinear Schr\"odinger equation. However, it is in fact the difference that gives the different conclusions of these two equations. Let $\psi=\psi(x)$ be a smooth real function.
Define
\begin{align}
I(t)&=\int_{\R}\psi |v(t)|^2\,dx;\\
J(t)&=2\mbox{Im}\int_{\R}\psi \bar v(t) v_x(t)\,dx+\frac 12\int
\psi|v(t)|^4\,dx.
\end{align}
\begin{lem}\label{virial-identity}
Let $v$ be the solution of \eqref{eqs:DNLS-under-=gauge1} with
$v(0)=v_0\in \H1$, and let $\psi\in C^3$. Then
\begin{align}
I'(t)&=2\mbox{Im}\int_{\R}\psi'\bar v(t) v_x(t)\,dx;\label{eqs:I'}\\
J'(t)&=4\int_{\R}\psi'\big(|v_x(t)|^2-\frac
1{16}|v(t)|^6\big)\,dx-\int_{\R}\psi'''|v(t)|^2\,dx.\label{eqs:J'}
\end{align}
\end{lem}
\begin{proof}
Employing the gauge transform
$$
w(t,x):=\mathcal G_{-\frac12}u(t,x)=\mathcal G_{\frac14}v(t,x),
$$
then by Lemma \ref{lem:Ga},  $w$ obeys the equation
$$
iw_t+w_{xx}=i|w|^2w_x.
$$
Moreover, since $v(t,x)=\mathcal G_{-\frac14}w(t,x)$, by \eqref{eqs:Gax},
$$
\partial_xv(t,x)=e^{-i\frac14\int_{-\infty}^x |w(t,y)|^2\,dy}\big(-\frac14 i|w|^2w+w_x\big).
$$
Thus, we have
\begin{align*}
I(t)=\int_{\R}\psi |w(t)|^2\,dx
\quad \mbox{and} \quad
J(t)=2\mbox{Im}\int_{\R}\psi \bar w(t) w_x(t)\,dx.
\end{align*}
Now by a direct computation, we get
\begin{align}
I'(t)&=2\mbox{Re}\int_{\R}\psi\bar w(t,x) \partial_tw(t,x)\,dx=2\mbox{Re}\int_{\R}\psi\bar w \big(i w_{xx}+|w|^2w_x\big)\,dx\notag\\
&=2\mbox{Im}\int_{\R}\psi'\bar w w_{x}\,dx-\frac12\int_{\R} \psi' |w|^4\,dx.\label{3710.18}
\end{align}
Applying \eqref{eqs:Gax} again,
\begin{align}\label{3710.20}
\partial_xw(t,x)=e^{\frac14i\int_{-\infty}^x |v(t,y)|^2\,dy}\big(\frac14 i|v|^2v+v_x\big).
\end{align}
This together with \eqref{3710.18} gives \eqref{eqs:I'}.
Now we turn to \eqref{eqs:J'}. For this, we get
\begin{align}
J'(t)&=2\mbox{Im}\int_{\R}\psi  \bar w_t(t,x)  w_x(t,x)\,dx+2\mbox{Im}\int_{\R}\psi\bar w(t,x) w_{xt}(t,x)\,dx\notag\\
&=-4\mbox{Im}\int_{\R}\psi   w_t \bar w_x\,dx-2\mbox{Im}\int_{\R}\psi'\bar w w_{t}\,dx\notag\\
&=-4\mbox{Im}\int_{\R}\psi \bar w_x  (iw_{xx}+|w|^2w_x) \,dx-2\mbox{Im}\int_{\R}\psi'\bar w (iw_{xx}+|w|^2w_x) \,dx\notag\\
&=-4\mbox{Re}\int_{\R}\psi \bar w_x  w_{xx} \,dx-2\mbox{Re}\int_{\R}\psi'\bar w w_{xx} \,dx-2\mbox{Im}\int_{\R}\psi' |w|^2\bar w w_x\,dx\notag\\
&=4\int_{\R}\psi'|w_x|^2\,dx+2\mbox{Re}\int_{\R}\psi''\bar w w_x\,dx-2\mbox{Im}\int_{\R}\psi' |w|^2\bar w w_x\,dx\notag\\
&=4\int_{\R}\psi'|w_x|^2\,dx-\int_{\R}\psi'''| w|^2\,dx-2\mbox{Im}\int_{\R}\psi' |w|^2\bar w w_x\,dx.\label{3710.19}
\end{align}
Now using \eqref{3710.20}, we have
$$
|w_x|^2=|v_x|^2+\frac12\mbox{Im}\big(|v|^2\bar v v_x\big)+\frac1{16}|v|^6;
$$
and
$$
|w|^2=|v|^2;\quad \mbox{Im}\big(|w|^2\bar w w_x\big)=\mbox{Im}\big(|v|^2\bar v v_x\big)+\frac14|v|^6.
$$
These insert into \eqref{3710.19} and we obtain \eqref{eqs:J'}.
\end{proof}

\section{A comparison between DNLS, NLS-5 and gKdV-5}

In this section, we discuss the nonlinear Schr\"odinger equation with derivative \eqref{eqs:DNLS-under-=gauge1},
the focusing, quintic nonlinear
Schr\"{o}dinger equation (NLS-5) which reads as
\begin{equation}\label{quintic-NLS}
i\partial_{t}u+\partial_{x}^{2}u+\frac3{16} |u|^4u=0,
\end{equation}
and the focusing, quintic generalized Korteweg-de Vries equation (gKdV-5),
\begin{equation}\label{gKdV-5}
    \partial_t u +  \partial^3_{x} u +\frac3{16}  \partial_x(u^5)=0.
 \end{equation}
The first two equations have the same standing wave solutions of $e^{it}Q$, and the last one has a traveling wave solution $Q(x-t)$.
These three equations have the same energies in the form of \eqref{energy-43}. So by the sharp Gagliardo-Nirenberg inequality, all of them are global well-posdness in $\H1$ when the initial data $\|u_0\|_{L^2}<\|Q\|_{L^2}=\sqrt {2\pi}$.

Now we continue to discuss the difference between the first equation (DNLS) and the last two (NLS-5, gKdV-5).

First of all, we give some products from Lemma \ref{virial-identity}. We always assume that
$v$ is smooth enough.
Taking $\psi=x, x^2$ respectively,  then by
\eqref{eqs:I'}, we have
$$
\frac{d}{dt}\int_{\R}x |v(t)|^2\,dx=2\mbox{Im}\int_{\R}\bar v(t) v_x(t)\,dx;
$$
and
\begin{equation}\label{virial-1}
\frac{d}{dt}\int_{\R}x^2 |v(t)|^2\,dx=4\mbox{Im}\int_{\R}x\bar v(t) v_x(t)\,dx,
\end{equation}
respectively. 
Note that these two identities resemble to the corresponding one of the mass-critical nonlinear
Schr\"{o}dinger equation \eqref{quintic-NLS}.

Now we take $\psi=1$ in \eqref{eqs:J'}, it gives the momentum conservation law,
\begin{equation}\label{mom-law}
P(v(t)):=\mbox{Im} \int_{\R}\bar v(t) v_x(t)\,dx+\frac 14\int_{\R}
|v(t)|^4\,dx=P(v_0).
\end{equation}
Then taking $\psi=x$, we have
\begin{equation}\label{virial-2}
\frac{d}{dt}\Big[2\mbox{Im}\int_{\R}x \bar v(t) v_x(t)\,dx+\frac 12\int_{\R}
x|v(t)|^4\,dx\Big]=4E(v_0).
\end{equation}
This equality is different from the situation of the mass-critical nonlinear
Schr\"{o}dinger equation \eqref{quintic-NLS}.
More precisely, for the solution $u$ of \eqref{quintic-NLS} with the initial data $u_0$,  we have
\begin{equation}\label{virial-2-NLS}
\frac{d}{dt}\Big[2\mbox{Im}\int_{\R}x \bar u(t) u_x(t)\,dx\Big]=4E(u_0).
\end{equation}
Compared with the identity \eqref{virial-2-NLS}, there is an additional term $\frac 12\int
x|v(t)|^4\,dx$ in \eqref{virial-2}. Indeed, for the solution of \eqref{quintic-NLS}, combining with the same identity in \eqref{virial-1}, one has
\begin{equation}\label{variance}
\frac{d^2}{dt^2}\int_{\R}x^2 |u(t)|^2\,dx=8E(u_0).
\end{equation}
But this does not hold for the solution of \eqref{eqs:DNLS-under-=gauge1}. The ``surplus'' term $\frac 12\int
x|v(t)|^4\,dx$ in \eqref{virial-2} breaks the convexity of the
variance. It is precisely this difference that leads to the distinct phenomena of the solutions of these two equations, at least at the technical level.

Using the virial identity \eqref{variance},  Glassey
\cite{Glassey77} proved that the solution $u$ of the mass-critical nonlinear
Schr\"{o}dinger equation
\begin{equation*}
   \aligned
    \partial_t u +  \Delta u +|u|^{\frac 4N}u=0, \quad (t,x)\in \R\times \R^N. \\
   \endaligned
 \end{equation*}
blows up in finite time when $u_0\in H^1(\R^N), xu_0\in L^2(\R^N)$ and $E(u_0)<0$. Further,
in the 1D case, Ogawa and Tsutsumi \cite{OgTs92} proved that the solutions of \eqref{quintic-NLS}
blow up in finite time when $u_0\in \H1$ and $E(u_0)<0$. See also \cite{DuWuZh-NLS, GlMe-95, HoRo2, Na99},  where all the solutions of the nonlinear Schr\"odinger equations with power nonlinearity blow up in finite time or infinite time if their energies are negative.
However, Theorem
\ref{thm:main1} depicts a different scene, where there exist global  and uniformly bounded solutions
even if $E(v_0)<0$.

The situation is also different from the mass-critical generalized
KdV equation \eqref{gKdV-5}. The latter has also virial identity
$$
\frac{d}{dt}\int_{\R}(x+t) |u(t)|^2\,dx=\int_{\R}u^2\,dx-3\int_{\R}|u_x|^2\,dx-\frac13\int_\R|u|^6\,dx.
$$
The blow-up of the solutions to \eqref{gKdV-5} also occurs when the initial data $u_0$ satisfies
$E(u_0)<0$, \eqref{small condition-further}, and some decay
conditions, see \cite{MaMe-2002-gKdV, Me-2001-gKdV}.

\section{Blow-up for the DNLS on the half line}

In this section, we use the virial identities obtained in Subsection 2.2 to study the blow-up solutions for the nonlinear Schr\"odinger equation with derivative on the half line. Consider the problem \eqref{DNLS-halfline}, and set 
\begin{equation*}
v(t,x)=e^{-i\frac34\int_{0}^x |u(t,y)|^2\,dy}u(t,x),
\end{equation*}
Using this gauge transformation, we see that $v$ is the solution of
\begin{equation}\label{8.19}
   \left\{ \aligned
    &i\partial_{t}v+\partial_{x}^{2}v=\frac i2|v|^2v_x-\frac i2v^2\bar{v}_x-\frac
     3{16}|v|^4v,\qquad t\in \R, x\in (0,+\infty),
    \\
    &v(0,x)  =v_0(x),
    \\
    &v(t,0)  =0.
   \endaligned
  \right.
 \end{equation}
Note that after replacing the integral domain $\R$ by $\R^+$, the energy conservation law and all of the virial identities obtained in Subsection 2.2 also hold true for $v$.

Now using the virial identities and Glassey's argument \cite{Glassey77}, we give the proof of Theorem \ref{thm:main2}.

\begin{proof}[Proof of Theorem \ref{thm:main2}]
Let $v$ be the solution to \eqref{8.19}.
Denote
$$
I(t)=\int_0^\infty x^2|v(t,x)|^2\,dx.
$$
Then, by the analogous identity as \eqref{virial-1}, we have
\begin{align*}
I'(t)=&4\mbox{Im}\int_0^\infty x\bar v(t) v_x(t)\,dx\\
=&2\Big[2\mbox{Im}\int_0^\infty x \bar v(t) v_x(t)\,dx+\frac 12\int_0^\infty\!\!\! x|v(t)|^4\,dx\Big]- \int_0^\infty\!\!\! x|v(t)|^4\,dx.
\end{align*}
Now, by the analogous identity as \eqref{virial-2}, we get
$$
\frac{d}{dt}\Big[2\mbox{Im}\int_0^\infty x \bar v(t) v_x(t)\,dx+\frac 12\int_0^\infty
x|v(t)|^4\,dx\Big]=4E(v_0).
$$
Therefore, using these two identities, we obtain
\begin{align*}
I''(t)=&8E(v_0)- \frac{d}{dt}\int_0^\infty\!\!\! x|v(t)|^4\,dx.
\end{align*}
Integrating in time twice, we have
\begin{align}
I(t)=&I(0)+I'(0)t+\int_0^{t}\!\!\int_0^s I''(\tau)\,d\tau ds\notag\\
=&I(0)+I'(0)t+\int_0^{t}\!\!\int_0^s\Big(8E(v_0)- \frac{d}{d\tau}\int_0^\infty\!\!\! x|v(\tau)|^4\,dx\Big)\,d\tau ds\notag\\
=&4E(u_0)t^2+\Big(I'(0)+\int_0^\infty\!\!\! x|v_0|^4\,dx\Big)t+I(0)-\int_0^{t}\!\!\int_0^\infty\!\!\! x|v(s)|^4\,dx ds\notag\\
\le &4E(u_0)t^2+\Big(I'(0)+\int_0^\infty\!\!\! x|v_0|^4\,dx\Big)t+I(0).\label{eqs:38}
\end{align}
Since $E(v_0)=E_D(u_0)<0$, there exists a finite time $T_*>0$ such that $I(T_*)=0,$
$$
I(t)>0, \,\,0<t<T_*,
$$
and
$$
I(t)=O(T_*-t),  \mbox{ as } t\to T_*-.
$$
Note that
\begin{align*}
\int_0^\infty |v_0(x)|^2\,dx&=\int_0^\infty |v(t,x)|^2\,dx=-2\mbox{Re} \int_0^\infty xv(t,x) \overline{v_x}(t,x)\,dx\\
&\le 2 \|xv(t,x)\|_{L^2_x(\R^+)}\|v_x(t,x)\|_{L^2_x(\R^+)}=2\sqrt{I(t)}\>\|v_x(t,\cdot)\|_{L^2(\R^+)}.
\end{align*}
Then there is a constant $C=C(v_0)>0$, such that
\begin{align}\label{Expo-v}
\|v_x(t,\cdot)\|_{L^2(\R^+)}\ge \frac{\int_0^\infty |v_0(x)|^2\,dx}{2\sqrt{I(t)}}
\ge \frac{C}{\sqrt{T_*-t}} \to \infty,
\end{align}
and the right-hand side goes to $\infty$ as $ t\to  T^*-.$
Therefore, $v(t)$ blows up at time $T_*<+\infty$. Since
$$
v_x=e^{-i\frac34 \int_{0}^x |u(t,y)|^2\,dy}\big(-i\frac34|u|^2u+u_x\big),
$$
by Gagliardo-Nirenberg inequality and mass conservation law, there exists $C=C(u_0)$ such that
\begin{align*}
\|v_x(t,\cdot)\|_{L^2(\R^+)}\le & \|u_x(t,\cdot)\|_{L^2(\R^+)}+\frac34\|u(t,\cdot)\|_{L^6(\R^+)}^3\le  C\|u_x(t,\cdot)\|_{L^2(\R^+)}.
\end{align*}
Thus by \eqref{Expo-v}, this gives the analogous estimate on $u$. 
\end{proof}

One may note from the proof that the key ingredient to obtain the blow-up result of initial boundary value problem on the half-line case is the positivity of the ``surplus'' term $\int_0^\infty\!\! x|v(t)|^4\,dx$. This is not true for the Cauchy problem.

\vspace{0.5cm}
\section{Proof of the  Theorem \ref{thm:main1}}
Let $(-T_-(u_0),T_+(u_0))$ be the maximal lifespan of the solution $u$ of \eqref{eqs:DNLS}. To prove  Theorem \ref{thm:main1}, it is  sufficient to
obtain the (indeed uniformly)\emph{ a priori} estimate of the
solutions on $H^1$-norm, that is,
$$
\sup\limits_{t\in (-T_-(u_0),T_+(u_0))}\|v_x(t)\|_{L^2}< +\infty.
$$
Now we argue by contradiction and suppose that there exists a
sequence $\{t_n\}$ with
$$
t_n\to -T_-(u_0),\quad \mbox{ or } \quad T_+(u_0),
$$
such that
\begin{equation}\label{infty-sequence}
\|v_x(t_n)\|_{L^2}\to +\infty,   \mbox{ as  } n\to\infty.
\end{equation}
Let
\begin{equation}\label{def_lambda_n}
\lambda_n=\|Q_x\|_{L^2}/\| v_x(t_n)\|_{L^2},
\end{equation}
and
\begin{equation}\label{eqs:scaling-wn}
w_n(x)=\lambda_n^{\frac 12}v(t_n, \lambda_nx).
\end{equation}

Then by \eqref{infty-sequence},
\begin{equation*}
\|\partial_x w_n\|_{L^2}=\|Q_x\|_{L^2},\quad \mbox{ and }
\quad\lambda_n\to 0, \mbox{ as } n\to \infty.
\end{equation*}
First we have the following lemma.

\begin{lem}\label{lem:var-claim} For any $\varepsilon>0$, there exists a small
$\varepsilon_*=\varepsilon_*(\varepsilon)>0$, such that if the
function $f\in \H1$ satisfies
$$
\int_{\R}|f(x)|^2\,dx <{2\pi}+\varepsilon_*,\quad \|\partial_x
f\|_{L^2}=\|\partial_xQ\|_{L^2},\quad E(f)<\varepsilon_*,
$$
then there exist $\gamma_0, x_0\in \R$, such that
$$
\|f-e^{-i\gamma_0}Q(\cdot-x_0)\|_{H^1}\le \varepsilon.
$$
\end{lem}

We put the  proof of this lemma \ref{lem:var-claim} at the end of this section and apply it to prove Theorem \ref{thm:main1}.
Let $\varepsilon_0>0$ be a fixed small constant which will be
decided later, and let
$\varepsilon_*=\varepsilon_*(\varepsilon_0)>0$ be the number defined in Lemma
\ref{lem:var-claim}. By \eqref{small condition-further},
\eqref{eqs:scaling-wn} and a simple computation,
$$
\int_{\R}|w_n(x)|^2\,dx=\int_{\R}|v_0(x)|^2\,dx <{2\pi}+\varepsilon_*,
$$
and
$$
\|\partial_x w_n\|_{L^2}=\|Q_x\|_{L^2},\quad E(w_n)=\lambda_n^2
E(v_0)\to 0.
$$
Then, by Lemma \ref{lem:var-claim}, we may inductively
construct the sequences $\{\gamma_n\}, \{x_n\}$ which satisfy
\begin{equation}\label{variantial-structure}
\|w_n-e^{-i\gamma_n}Q(\cdot-x_n)\|_{H^1}\le \varepsilon_0 \quad \mbox{for any}\,\, n\ge n_0,
\end{equation}
where $n_0=n_0(\varepsilon_0)$ is a positive large number.
Let
$$
\varepsilon(t_n,x)=e^{i\gamma_n}w_n(x+x_n)-Q.
$$
Then
\begin{equation}\label{10-1}
w_n(x)=e^{-i\gamma_n}Q(x-x_n)+e^{-i\gamma_n}\varepsilon(t_n,x-x_n).
\end{equation}
Therefore,  by \eqref{eqs:scaling-wn}, \eqref{10-1}, and
\eqref{variantial-structure}, we have
\begin{equation}\label{varepsilon-w}
\aligned
  v(t_n,x)&=e^{-i\gamma_n}\lambda_n^{-\frac
  12}(\varepsilon+Q)(t_n,\lambda_n^{-1}x-x_n),\quad \|\varepsilon(t_n)\|_{H^1}\le \varepsilon_0.
  \endaligned
\end{equation}
By the momentum and \eqref{varepsilon-w}, one has
\begin{align*}
  P(v(t_n))=&\mbox{Im} \int_{\R}\bar v(t_n) v_x(t_n)\,dx+\frac 14\int_{\R}
|v(t_n)|^4\,dx\notag\\
=&\lambda_n^{-2}\mbox{Im} \int_{\R}\big(\bar
\varepsilon(t_n)+Q\big)(t_n, \lambda_n^{-1}x-x_n)\cdot\big(\varepsilon_x(t_n)+Q_x\big)(t_n, \lambda_n^{-1}x-x_n)\,dx\notag\\
&+\frac
14\lambda_n^{-2}\int_{\R}\big|(\varepsilon(t_n)+Q)(t_n, \lambda_n^{-1}x-x_n)\big|^4\,dx\notag\\
=&\lambda_n^{-1}\mbox{Im} \int_{\R}\big(\bar
\varepsilon(t_n)+Q\big)\big(\varepsilon_x(t_n)+Q_x\big)\,dx+\frac
14\lambda_n^{-1}\int_{\R}|\varepsilon(t_n)+Q|^4\,dx\notag\\
=&\lambda_n^{-1}\Big[\frac 14\|Q\|_{L^4}^4+\mbox{Im} \int_{\R}
\big(Q_x\varepsilon(t_n)+Q\varepsilon_x(t_n)+\bar \varepsilon
\varepsilon_x(t_n)\big)\,dx\notag\\
&+\frac 14\int_{\R}\big(|\varepsilon+Q|^4-Q^4\big)\,dx \Big]\notag\\
=&\lambda_n^{-1}\cdot\Big(\frac
14\|Q\|_{L^4}^4+O(\|\varepsilon(t_n)\|_{H^1})\Big)\notag\\
\ge &\lambda_n^{-1}\cdot\Big(\frac
14\|Q\|_{L^4}^4-C\varepsilon_0\Big).
\end{align*}
Thus, by choosing $\varepsilon_0$ small enough such that
$C\varepsilon_0\le \frac 1{8}\|Q\|_{L^4}^4$, one has
$$
  P(v(t_n))\ge \lambda_n^{-1}\cdot\frac 18\|Q\|_{L^4}^4.
$$
By the momentum conservation law, this proves that
$
P(v_0)\lambda_n\ge \frac 18 \|Q\|_{L^4}^4.
$
That is, by \eqref{def_lambda_n},
\begin{equation}\label{upper-bound}
    \|v_x(t_n)\|_{L^2}\le 8 P(v_0)\|Q_x\|_{L^2}/\|Q\|_{L^4}^4.
\end{equation}
This violates \eqref{infty-sequence}. Therefore, we prove that there exists $C_0=C_0(\varepsilon_*,\|v_0\|_{H^1})$, such that
$$
\sup\limits_{t\in \R}\|v_x(t)\|_{L^2}\le  C_0.
$$
Now, for the solution $u$ of \eqref{eqs:DNLS} (with $\lambda=1$), we have $u=\mathcal G_{\frac34}v$. Thus, by \eqref{eqs:Gax}, we have
$$
u_x=e^{i\frac34 \int_{-\infty}^x |v(t,y)|^2\,dy}\big(i\frac34|v|^2v+v_x\big).
$$
Therefore, by \eqref{sharp Gagliardo-Nirenberg} and mass conservation law, for any $t\in \R$,
\begin{align*}
\|u_x(t)\|_{L^2}\le & \|v_x(t)\|_{L^2}+\frac34\|v(t)\|_{L^6}^3
\le  \|v_x(t)\|_{L^2}+\frac3{2\pi}\|v(t)\|_{L^2}^2\|v_x(t)\|_{L^2}\\
\le & C_0\big(1+\frac3{2\pi}\|u_0\|_{L^2}^2\big).
\end{align*}
Thus we finish the proof of
Theorem \ref{thm:main1}.

\begin{proof}[Proof of the Lemma \ref{lem:var-claim}] The proof may follow from the standard
variational argument, see \cite{Me-2001-gKdV, W2} for examples,
see also e.g. \cite{Ba-ASNSP-04, HmKe-IMRN-05} for its applications.
Here
we prove it by using the profile decomposition (see \cite{Garard-98}
for example) for sake of the completeness. Let
$\{f_n\}\subset H^1(\R)$ be any sequence satisfying
$$
\|f_n\|_{L^2}\to \|Q\|_{L^2},\quad \|\partial_x
f_n\|_{L^2}=\|Q_x\|_{L^2},\quad E(f_n)\to 0.
$$
Then by the profile decomposition, there exist $\{V^j\}, \{x_n^j\}$
such that, up to a subsequence,
$$
f_n =\sum\limits_{j=1}^L V^j(\cdot-x_n^j)+R_n^L,
$$
where $|x_n^j-x_n^k|\to \infty, \mbox{ as } n\to\infty, j\neq k$, and
\begin{equation}\label{Min-Norm-H1}
\aligned
  &\lim\limits_{L\to\infty}\left[\lim\limits_{n\to\infty}\|R_n^L\|_{L^6}\right]=0.
  \endaligned
\end{equation}
Moreover,
\begin{equation}\label{Sum-Norm-H1}
\aligned
  \|f_n\|_{H^s}^2&=\sum\limits_{j=1}^L \|V^j\|_{H^s}^2+\|R_n^L\|_{H^s}^2+o_n(1),  \mbox{ for } s=0,1,\\
  E(f_n)&=\sum\limits_{j=1}^L E(V^j)+E(R_n^L)+o_n(1).
  \endaligned
\end{equation}
Since $\|f_n\|_{L^2}\to \|Q\|_{L^2}$, one has by
\eqref{Sum-Norm-H1},
\begin{equation}\label{10-2}
\|V^j\|_{L^2}\le \|Q\|_{L^2}, \mbox{ for any } j\ge 1.
\end{equation}
This implies by sharp Gagliardo-Nirenberg inequality \eqref{sharp
Gagliardo-Nirenberg} that
$
E(V^j)\ge 0 \mbox{ for any } j\ge 1.
$
Further, by \eqref{Min-Norm-H1}, one has
$$
\lim\limits_{L\to\infty}\left[\lim\limits_{n\to\infty}E(R_n^L)\right]\ge
0.
$$
Since $E(f_n)\to 0$, we have
$
E(V^j)=0$ for any $j\ge 1.
$
Combining with \eqref{10-2} and \eqref{sharp
Gagliardo-Nirenberg}, this again yields that
$$
\|V^j\|_{L^2}=\|Q\|_{L^2}, \quad\mbox{ or } \quad V^j=0.
$$
Since $\|f_n\|_{L^2}\to \|Q\|_{L^2}$, there exactly exists one $j$, let $j=1$ such that
$$
\|V^1\|_{L^2}=\|Q\|_{L^2}, \quad V^j=0 \mbox{ for any } j\ge 2.
$$
Moreover, by \eqref{Sum-Norm-H1} and
\eqref{sharp Gagliardo-Nirenberg}, when $n\to \infty$, we have
$
R_n^L\to 0$  in  $ L^2(\R)$, and then further in $\H1$.
Therefore,
$$
\|\partial_x V^1\|_{L^2}=\|Q_x\|_{L^2},\quad E(V^1)= 0,
$$
and
$
f_n\to V^1 \mbox{ in }\H1,$ as $ n\to \infty
$. 
Now we note that $V^1$ attains the sharp
Gagliardo-Nirenberg inequality \eqref{sharp Gagliardo-Nirenberg},
thus by the uniqueness of the minimizer of Gagliardo-Nirenberg
inequality (see \cite{W}), we have
$
V^1= e^{-i\gamma_0}Q(\cdot-x_0)$,  for some $\gamma_0\in
\R, x_0\in \R$.
This proves the lemma.
\end{proof}

\section*{Acknowledgements} The author thanks Professors Yongsheng Li and Changxing Miao for their help in writing this paper and for
many valuable suggestions. The author is also grateful to the anonymous referees for helpful comments.

\end{document}